\documentclass{amsart}

\usepackage{graphicx}
\usepackage[latin1]{inputenc}

\newtheorem{thm}{Theorem}[section]
\newtheorem{cor}[thm]{Corollary}
\newtheorem{lem}[thm]{Lemma}

\newtheorem{exa}[thm]{Example}
\theoremstyle{definition}
\newtheorem{dfn}[thm]{Definition}
\theoremstyle{remark}
\newtheorem{rem}[thm]{Remark}
\numberwithin{equation}{section}

\begin{document}

\title[]{On links in $S_{g} \times S^{1}$ and its invariants}%
\author{S.Kim}

\address{S.Kim, Bauman Moscow State Technical University, Moscow institute of physics and technology}%
\email{ksj19891120@gmail.com}%


\maketitle

\begin{abstract}
A virtual knot, which is one of generalizations of knots in $\mathbb{R}^{3}$ (or $S^{3}$), is, roughly speaking, an embedded circle in thickened surface $S_{g} \times I$. In this talk we will discuss about knots in 3 dimensional $S_{g} \times S^{1}$. We introduce basic notions for knots in $S_{g} \times S^{1}$, for example, diagrams, moves for diagrams and so on. For knots in $S_{g} \times S^{1}$ technically we lose over/under information, but we will have information how many times the knot rotates along $S^{1}$. We will discuss the geometric meaning of the rotating information and how to construct invariants by using the ``rotating'' information.
\end{abstract}

\section{Introduction}
One of generalizations of classical knot theory is {\em virtual knot theory}.
\begin{dfn}
{\em A virtual link} is an equivalence class of virtual knot diagrams modulo generalized Reidemeister moves described in Fig~\ref{vir_moves}. That is,

$$\{Virtual~link\} =\{ Virtual~link~diagrams\}/\langle moves \rangle$$
\end{dfn}

\begin{figure}
\centering\includegraphics[width=200pt]{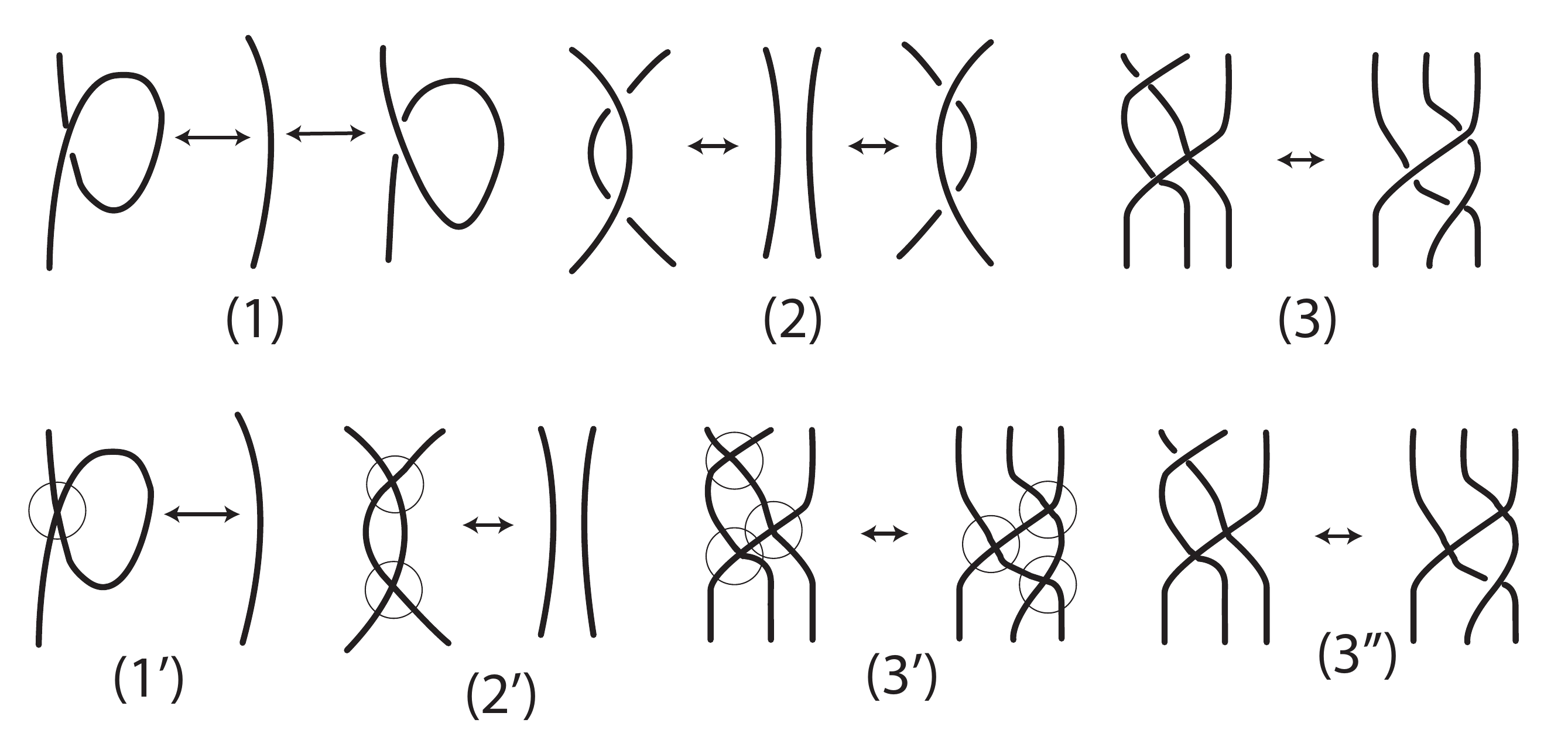} 
\caption{Generalized Reidemeister moves} \label{vir_moves}
 \end{figure}

It is well-known that the virtual links can be considered as links in a thickened surface $S_{g} \times [0,1]$ up to stabil/destabilization.

\begin{dfn}
{\em A virtual link} is a smooth embedding $L$ of a disjoint union of $S^{1}$ into $S_{g} \times [0,1]$. Each image of $S^{1}$ is called {\em a component} of $L$. A link of one component is called {\em a virtual knot}.

\end{dfn}

\begin{dfn}
  Let $L$ and $L'$ be two virtual links. If $L'$ can be obtained from $L$ by diffeomorshisms and stabil/destabilization of $S_{g} \times [0,1]$, then we call $L$ and $L'$ are \textit{equivalent}.
\end{dfn}

In virtual knot theory, by {\em a parity} defined by V.O. Manturov many invariants for classical knots are non trivially extended to virtual knots and it gives several interesting geometrical properties, for details, see \cite{Manturov_vir_book}. But, the extended invariants cannot show us something new property of classical knots, because every parity for classical knots is trivial.

In \cite{CM} M. Chrisman and V.O. Manturov studied virtual knots by using 2-component link $K \sqcup J$ with $lk(K,J)=0$, where $J$ is a fibered knot. Roughly speaking, if $J$ is a fibered knot, $S^{3} \backslash N(J)$ is homeomorphic to $\Sigma_{J} \times S^{1}$ where $\Sigma_{J}$ is a Seifert surface of $J$ and $K$ can be considered as a knot in $\Sigma_{J} \times S^{1}$. If $lk(K,J)=0$, then there exists a lifting $\hat{K} \subset \Sigma_{J} \times (0,1) \subset \Sigma_{J} \times [0,1]$ along the covering $p: \Sigma_{J} \times (0,1) \cong \Sigma_{J} \times \mathbb{R} \rightarrow \Sigma_{J} \times S^{1}$ defined by $p(x,r) = (x,e^{2\pi r})$. Then $\hat{K} \subset \Sigma_{J} \times [0,1]$ is placed in a thickened surface, that is, it can be considered as a virtual knot. Moreover, in \cite{CM} it is proved that the $\hat{K}$ is well-defined, that is, if $K \sqcup J$ and $K' \sqcup J'$ are equivalent in $S^{3}$, then $\hat{K}$ and $\hat{K}'$ are equivalent as virtual knots. But, there is a question: what happens if $lk(K,J) \neq 0$?

In this paper we discuss links in $S_{g} \times S^{1}$. The paper consists of 4 sections. In section 2, we will introduce links in $S_{g} \times S^{1}$ and its diagrams. And then we will define a ``labeling'' for each classical crossing of knots in $S_{g} \times S^{1}$, which shows that the half of a knot at each crossing rotates along $S^{1}$. In section 3, we introduce how to obtain a lifting in $S_{g} \times \mathbb{R}$ of a knot in $S_{g} \times S^{1}$. In section 4, we shortly discuss about relation links in $S^{3}$ and links in $S_{g} \times S^{1}$. In section 5, we will discuss how to construct invariants for knots in $S_{g} \times S^{1}$ by using the labels at crossings.

\section{Links in $S_{g} \times S^{1}$ and its diagram}

\begin{dfn}
{\em A link $L$ in $S_{g} \times S^{1}$} is a smooth embedding $L$ of a disjoint union of $S^{1}$ into $S_{g} \times S^{1}$. Each image of $S^{1}$ is called {\em a component} of $L$. A link of one component is called {\em a knot in $S_{g} \times S^{1}$}.

\end{dfn}

\begin{dfn}
  Let $L$ and $L'$ be two links in $ S_{g} \times S^{1}$. If $L'$ can be obtained from $L$ by diffeomorshisms and stabil/destabilization of $S_{g} \times S^{1}$, then we call $L$ and $L'$ are \textit{equivalent}.
\end{dfn}

The destabilization for $S_{g} \times S^{1}$, we mean the following; \\
Let $C$ be a non-contractible circle on the surface $S_{g}$ such that there exists a torus $T$ homotopic to the torus $C \times S^{1}$ and not intersecting the link. Then our destabilization is cutting of the manifold $S_{g} \times S^{1}$ along the torus $C \times S^{1}$ and pasting of two newborn components by $D \times S^{1}$.

Assume $x_{0} \in S^{1}$ is a point such that $S_{g} \times \{x_{0}\} \cap L(S^{1})$ is a set of finite points with no transversal points. 

\begin{figure}[h]
\begin{center}
 \includegraphics[width = 8cm]{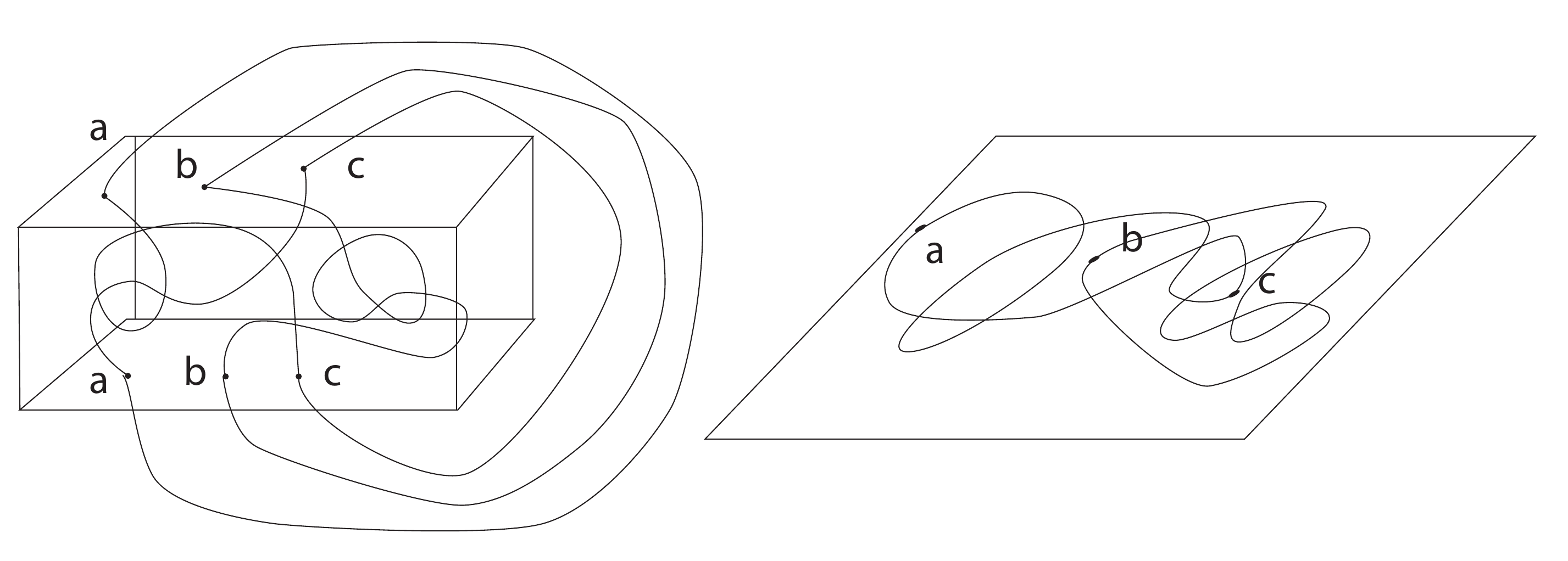}

\end{center}

\caption{}
\end{figure}

Assume that counterclockwise orientation is given on $S^{1}$.
Then $cl (S_{g} \times S^{1} )\backslash$  $(S_{g} \times \{x_{0}\})$ $\cong_{h}$ $S_{g} \times [0,1]$ where $h$ is an orientation preserving diffeomorphism.
Let $M_{L}$ $=$ $\overline{(S_{g} \times S^{1}) - (S_{g} \times \{x_{0}\})}$. Then $L$ in $M_{L}$ has a diagram on the surface $S_{g}$. The diagram of $L$ in $M_{L}$ has $n$-arcs with $n$ vertices and $m$-circles. Two arcs near to a vertex are corresponding to arcs near $S_{g} \times \{0\}$ and  $S_{g} \times \{1\}$, respectively. We change a vertex to two small lines such that if one of the lines is corresponding to an arc which is near to $S_{g} \times \{1\}$, the line is longer than another, see Figure~\ref{Vertex}.

\begin{figure}[h!]
\begin{center}
 \includegraphics[width = 4cm]{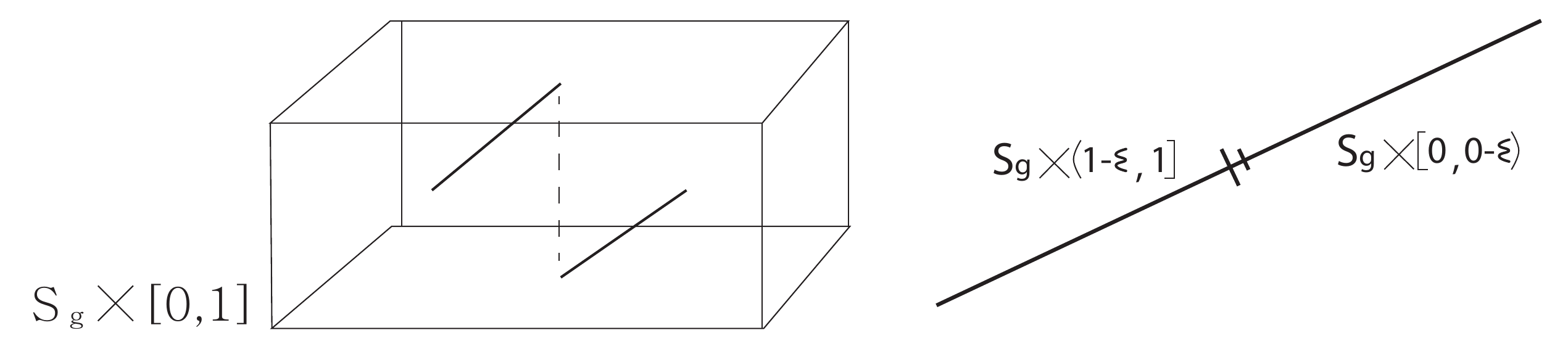}

\end{center}
\caption{}\label{Vertex}
\end{figure}

We will call this a \textit{diagram} of $L$ in $S_{g} \times S^{1}$ on $S_{g}$. For any diagrams with vertices on $S_{g}$, we can easily get a link $L$ in $S_{g} \times S^{1}$.

\begin{lem}\label{moves}
  Let $L$ and $L'$ be two links in $S_{g} \times S^{1}$. Let $D_{L}$ and $D_{L'}$ be diagrams of $L$ and $L'$ on $S_{g}$, respectively. Then $L$ and $L'$ are equivalent if and only if $D_{L'}$ can be obtained from $D_{L}$ by applying the following moves in Fig.~\ref{moves1}.

  \begin{figure}[h!]
\begin{center}
 \includegraphics[width = 8cm]{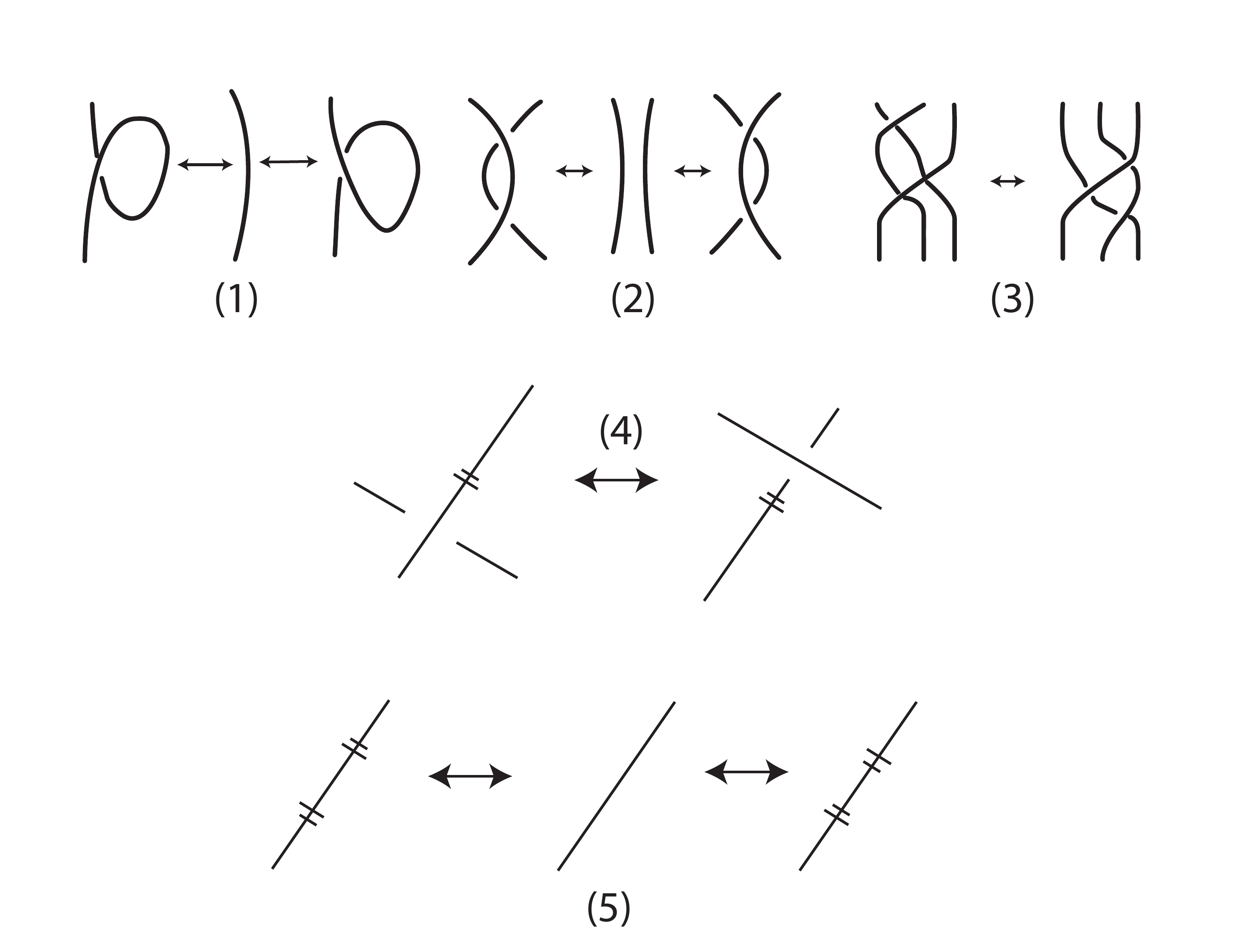}

\end{center}

\caption{}\label{moves1}
\end{figure}

\end{lem}

A link $L$ in $M_{L}$ has a diagram on the plane with {\em virtual crossings} as virtual links. The following theorem also holds.
\begin{thm}\label{thm:diag_on_plane}
   Let $L$ and $L'$ be two links in $S_{g} \times S^{1}$. Let $D_{L}$ and $D_{L'}$ be diagrams of $L$ and $L'$ on the plane, respectively. Then $L$ and $L'$ are equivalent if and only if $D_{L'}$ can be obtained from $D_{L}$ by applying the following moves in Fig.~\ref{moves2}.

  \begin{figure}[h!]
\begin{center}
 \includegraphics[width = 12cm]{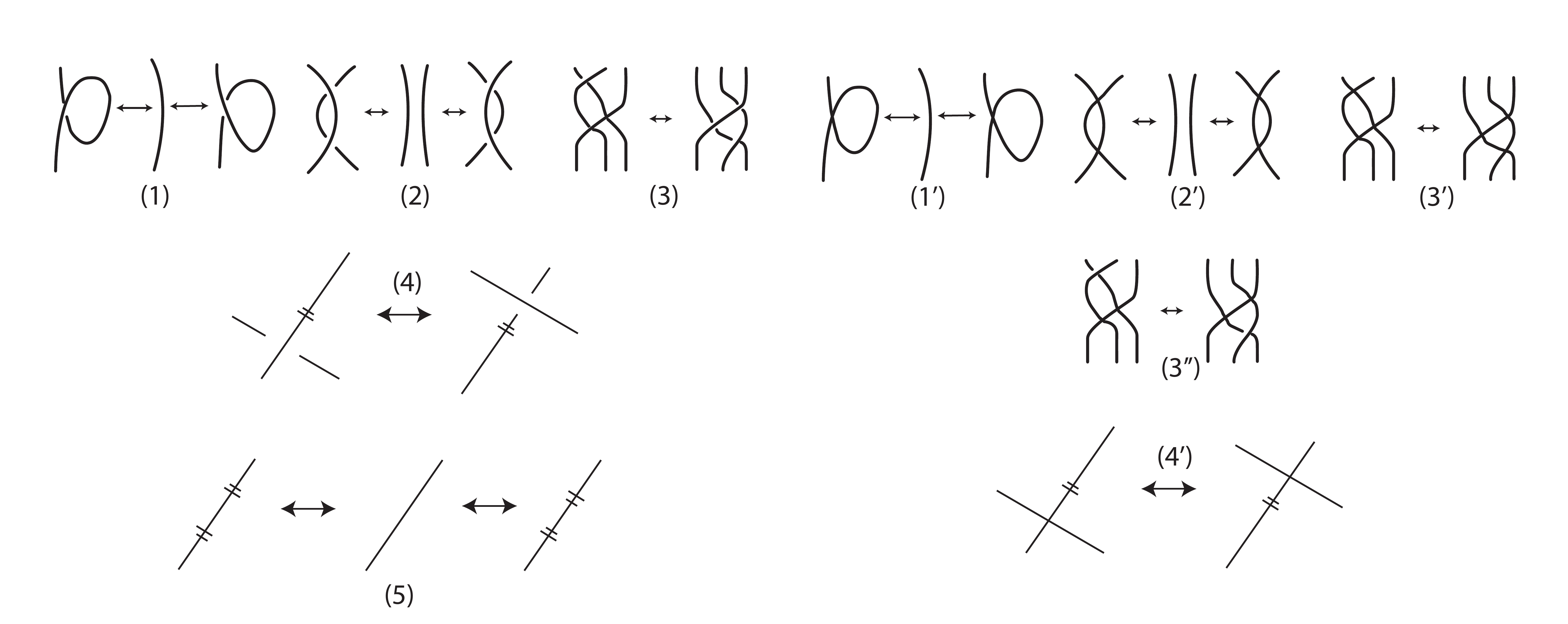}

\end{center}

 \caption{}\label{moves2}
\end{figure}
\end{thm}

One can find proofs for the previous lemma and theorem in \cite{}.
On the base of Theorem~\ref{thm:diag_on_plane} we can study knots by means of diagrams modulo local moves.

\subsection{Degree of knots in $S_{g}\times S^{1}$}
The most important information from knots in $S_{g}\times S^{1}$ is ``how many times the knot rotates along $S^{1}$''. More precisely, we consider the natural covering $\Pi: \mathbb{R} \rightarrow S^{1}$ defined by $\Pi(r) = e^{2\pi r i}$. Then the function $Id_{S_{g}} \times \Pi : S_{g}\times \mathbb{R} \rightarrow S_{g}\times S^{1}$ is also a covering over $S_{g}\times S^{1}$ where $Id_{S_{g}} : S_{g} \rightarrow S_{g}$ is the identity map.
 Let $\tilde{K}$ be a lifting of $K$ into $S_{g} \times \mathbb{R}$ along a covering $Id_{S_{g}} \times \Pi : S_{g}\times \mathbb{R} \rightarrow S_{g}\times S^{1}$. For a line segment $l$ of a diagram $D$ of $K$, there is a line segment $l'$ in $\tilde{K}$ in $S_{g} \times \mathbb{R}$. Let $\phi_{2} : S_{g} \times \mathbb{R} \rightarrow \mathbb{R}$. If $\phi_{2}(l') \in [a, a+1)$, then give a label $a$ to $l$. We consider the label $a$ as an element of $\mathbb{Z}$.


\begin{rem}\label{remark}
  For labels $a,b$ in the following figure, $a=b+1$.
  \begin{figure}[h!]
\begin{center}
 \includegraphics[width = 3cm]{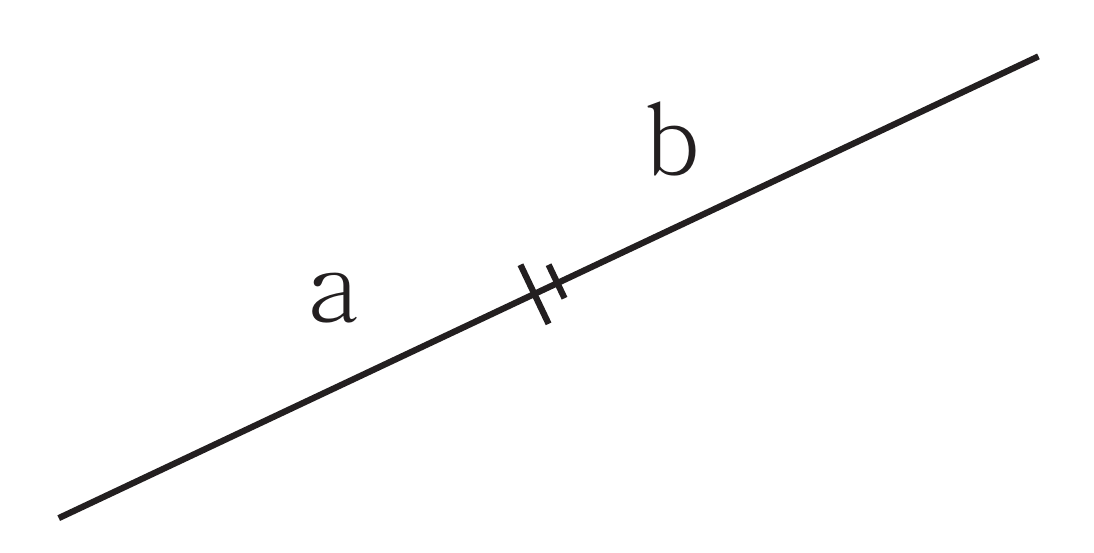}

\end{center}
 \caption{}
\end{figure}
\end{rem}

For each crossing, if over-arc is labeled by $b$ and under-arc is labeled by $a$ for some $a,b \in \mathbb{Z}$, then give a label $i = b - a$ to the crossing where $b-a$ is in $\mathbb{Z}$. Then we call $D$ with labeling for each classical crossing a \textit{labeled diagram}.
\begin{figure}[h!]
\begin{center}
 \includegraphics[width = 3cm]{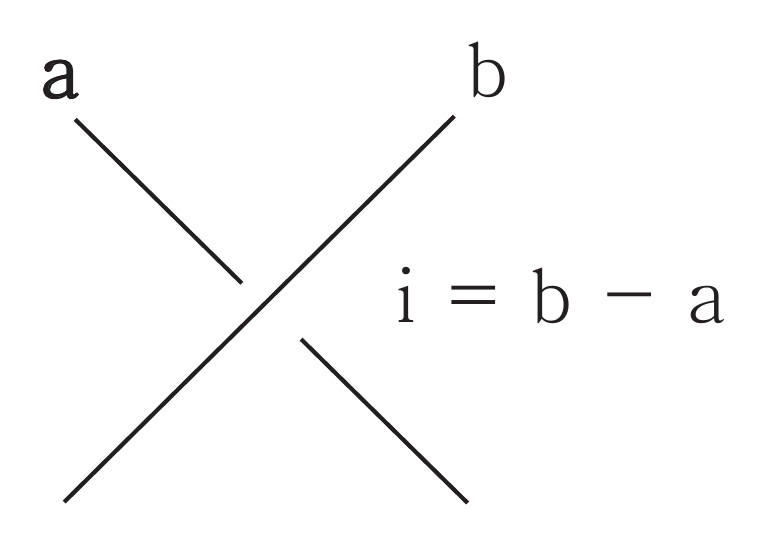}

\end{center}
 \caption{}
\end{figure}
\begin{thm}
  Let $K$ and $K'$ be two knots in $S_{g} \times S^{1}$ with $deg(K) = deg(K')$. Let $D_{K}$ and $D_{K'}$ be labeled diagrams of $K$ and $K'$, respectively. Then $K$ and $K'$ are equivalent if and only if $D_{K'}$ can be obtained from $D_{K}$ by applying the moves in Fig.~\ref{labelmoves2}.

  \begin{figure}[h!]\label{labelmoves2}
\begin{center}
 \includegraphics[width = 12cm]{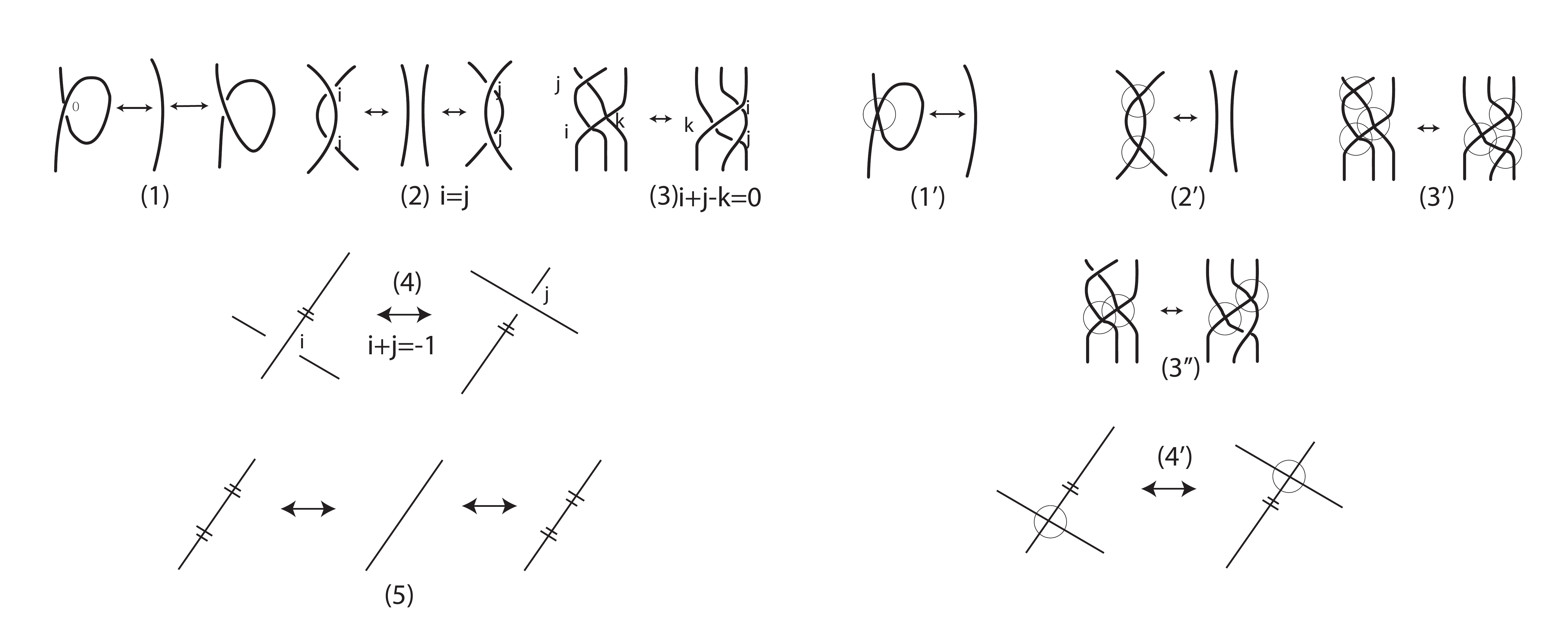}

\end{center}
 \caption{}
\end{figure}

\end{thm}

\begin{proof}
  Let $K$ and $K'$ be two knots in $S_{g} \times S^{1}$ with $deg(K) = deg(K')$. Let $D_{K}$ and $D_{K'}$ be labeled diagrams of $K$ and $K'$, respectively. By lemma~\ref{moves}, the moves are sufficient to present the equivalence of $K$ and $K'$. We focus on the relations among the labels of classical crossings related to each move. Note that arcs in a diagram corresponding to one arc in $M_{L}$ have the same label. For the move 1, it is related to only one arc. By the definition of labeling for classical crossings, the labeling of the classical crossing of the move 1 is $0$. For the move 2, this is related to two arcs which is labeled by $a$ and $b$, respectively. If the arc labeled by $b$ is an over arc, then two classical crossings are labeled by $i=b-a$. If the arc labeled by $a$ is an over arc, then two classical crossings are labeled by $j=a-b$. Then $i+j=0$. For the move 3, assume that the top arc is labeled by $a$, the middle arc is labeled by $b$ and the bottom arc is labeled by $c$. Let $i$, $j$ and $k$ be labels of crossings between the top arc and the middle arc, the middle arc and the bottom arc and the top arc and the bottom arc, respectively. Then $i =a-b$, $j=b-c$ and $k = a-c$. We can get $i+j-k = 0$. For the move 4; note that near a vertex, if two arcs have label $a$ and $b$ respectively, then $b-a=1$ by Remark \ref{remark}. Since the moving arc moves in $M_{K}$, the label of the arc is not changed. If $b-a=1$, then $i = a-c$ and $j = c-b = c-a-1$. Therefore $i+j=-1$.
\end{proof}

\begin{rem}
Geometrically, the label of a crossing $c$ means how many times the curve from $c$ to $c$ turns around $S^{1}$.
\end{rem}

\begin{rem}

If we consider the indices for classical crossing modulo $2$, then it becomes the parity. If we separate classical crossings by index $0$ and others, then it is a weak parity.
\end{rem}

\section{Knots in $S_{g} \times S^{1}$ and their liftings}

\subsection{Knots in $S_{g} \times S^{1}$ with degree $0$}

Let $K$ be an oriented knot in $S_{g} \times S^{1}$ with degree $0$. Then there exists a lifting $\hat{K}$ to $S_{g} \times \mathbb{R}$. Since $K$ has the degree $0$, $\hat{K}$ is a knot in $S_{g} \times \mathbb{R}$. In this section we construct a diagram of the lifting $\hat{K}$ to $S_{g} \times \mathbb{R}$. The algorithm is as follows:\\
\textbf{Step 1.} Let $D$ be an oriented diagram of $K$ in $S_{g} \times S^{1}$. Let fix a point point on a diagram and give a label for each arcs according to double lines. Let us say that the minimal label is $m$ and the maximal label is $M$.\\

\textbf{Step 2.} Let us make $M-m+1$ parallel planes placed vertically. Give numberings from bottom to top by integers from $m$ to $M$. Draw $M-m$ copies of $D$ for each plane. For a copy of a diagram on the plane with $k$ erase arcs which are not labeled by $k$. \\

\textbf{Step 3.} Let us start walking from the point on the diagram on the plane with number $0$, which corresponds to the fixed point. When we meet the double line, if it is longer line, then we connect the arc to the arc on the plane with number $1$, but if it is shorter line, then we connect the arc to the arc on the plane with number $-1$. Let us denote the obtained diagram by $\hat{D}$.

\begin{rem}
From the obtained diagram $\hat{D}$ one can easily obtain a knot in $S_{g}\times \mathbb{R}$.
\end{rem}

\begin{thm}
Let $D$ and $D'$ be two oriented diagrams in $S_{g} \times S^{1}$. If they are equivalent, then $\hat{D}$ and $\hat{D'}$ are equivalent as knots in $S_{g} \times \mathbb{R}$.
\end{thm}

\begin{proof}

Assume that $D'$ is obtained from $D$ by applying one of the moves in Fig.~\ref{moves2}. If $D'$ is obtained from $D$ by applying (1), (2), (3), (1'), (2'), (3') and (3''), then it is easy to see that $\hat{D'}$ can be obtained from $\hat{D}$ by applying virtual Reidemeister moves.

Suppose that $D'$ is obtained from $D$ by applying (4). The line segments in move (4) have labels $a,a+1$ and $b$ as described in Fig.~\ref{lifting4_type1} and Fig.~\ref{lifting4_type2}. As shown in Fig.s, when we lift diagrams, the difference is just a change of the place of a classical crossing and it follows that $\hat{D} \cong \hat{D}'$. Analogously one can show that $\hat{D} \cong \hat{D}'$ when $D'$ is obtained from $D$ by applying (4').

  \begin{figure}[h!]
\begin{center}
 \includegraphics[width = 8cm]{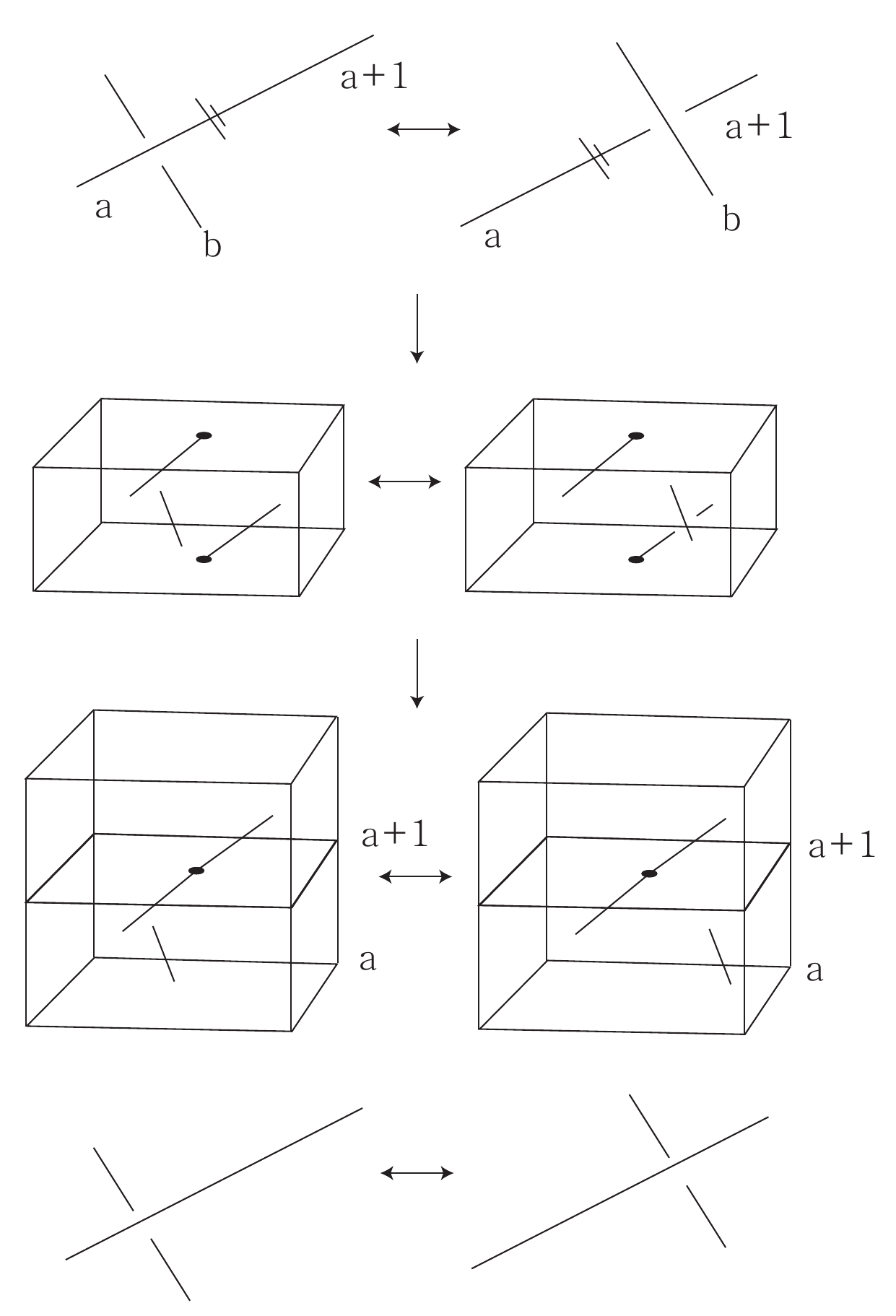}

\end{center}
 \caption{ $a\geq b$}\label{lifting4_type1}
\end{figure}

  \begin{figure}[h!]
\begin{center}
 \includegraphics[width = 8cm]{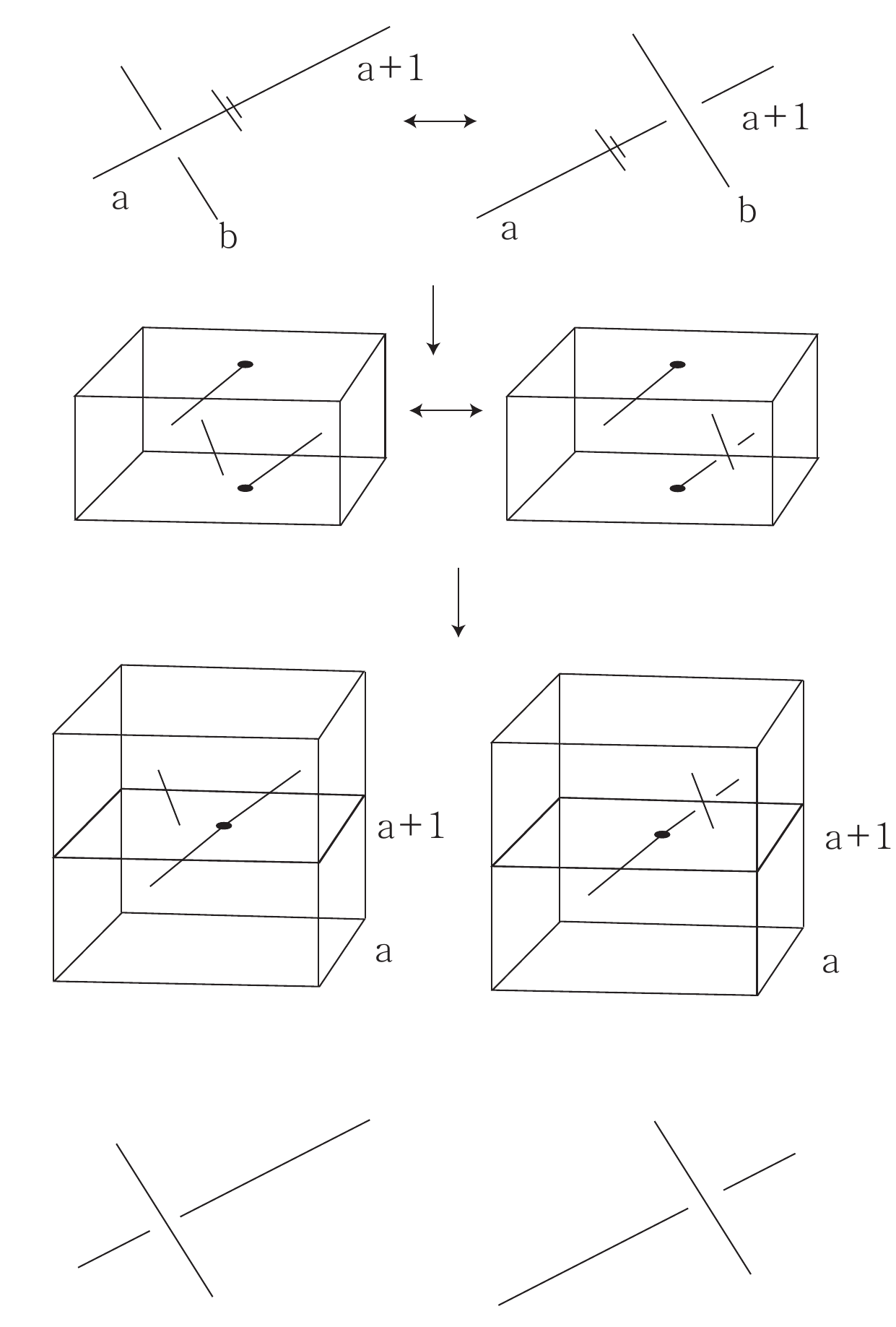}

\end{center}
 \caption{ $a< b$}\label{lifting4_type2}
\end{figure}

Suppose that $D'$ is obtained from $D$ by applying (5), see Fig.~\ref{lifting5}. When we lift the link to $S_{g} \times \mathbb{R}$ the line segment between two double-lines and others are placed in the different levels. Note that in $S_{g} \times \mathbb{R}$ under the line segment corresponding to the line segment between two double-lines there are no other arcs, so we can push it down as described under of Fig.~\ref{lifting5}. That is, $\hat{D} \cong \hat{D}'$ and the proof is completed.

  \begin{figure}[h!]
\begin{center}
 \includegraphics[width = 8cm]{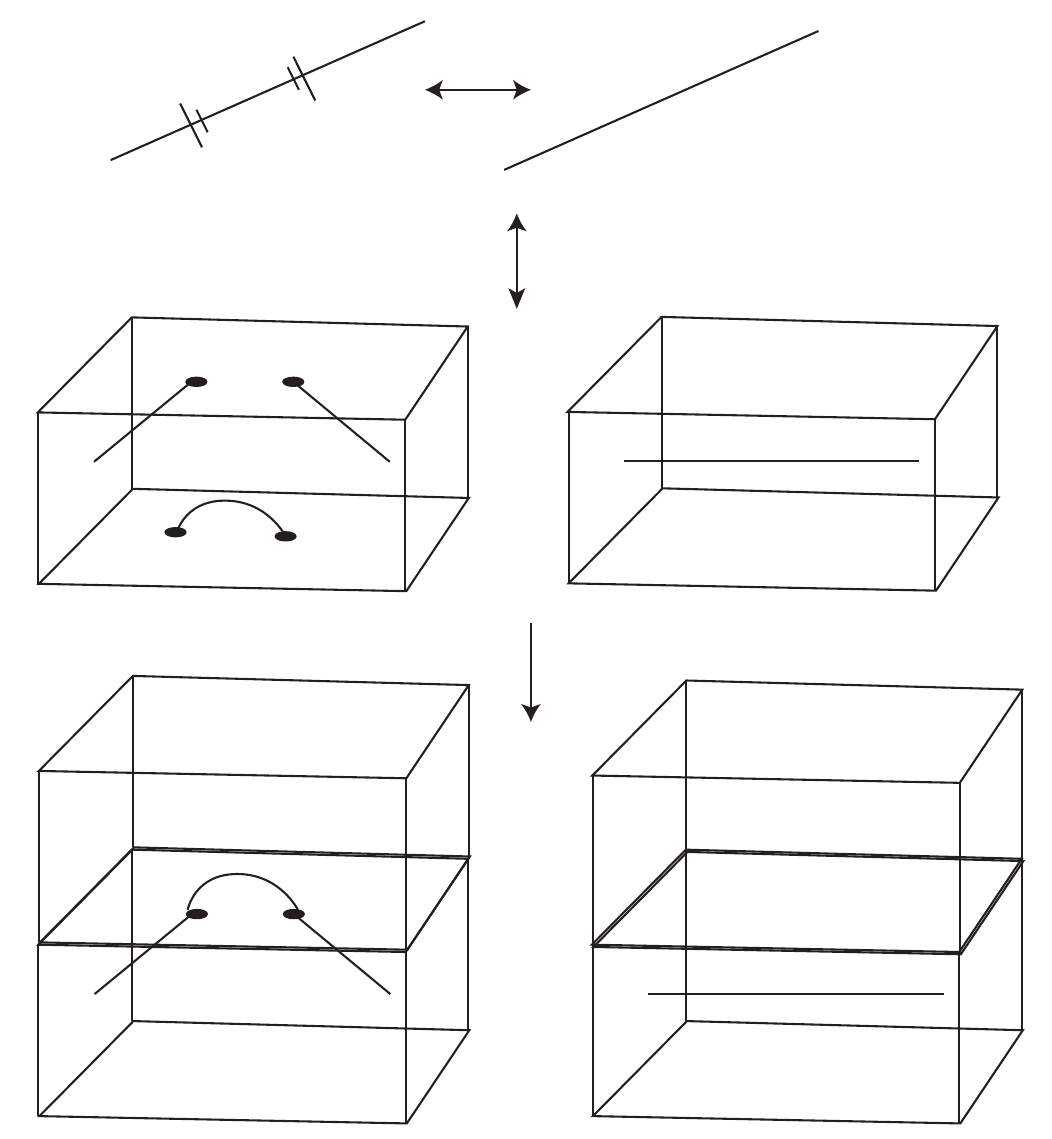}

\end{center}
 \caption{ }\label{lifting5}
\end{figure}

\end{proof}

\begin{rem}

If we consider liftings $\hat{K}_{s}$ such that $\hat{K}_{s}(0) = \hat{K}_{s}(0) = (x, s)$, then we obtain a link of infinitely many components. Note that the classical crossing with label $b-a$, then it corresponds to the crossing of infinite link $\sqcup_{s \in \mathbb{Z}} K_{s}$ between two components $K_{a}$ and $K_{b}$.

\end{rem}

\begin{cor}
Let $K$ and $K'$ be knots in $S_{g} \times S^{1}$. If $K$ and $K'$ are equivalent, then the liftings $\sqcup_{s \in \mathbb{Z}} K_{s}$ and $\sqcup_{s \in \mathbb{Z}} K'_{s}$ are equivalent in $S_{g} \times \mathbb{R}$.  
\end{cor}

\subsection{Knots in $S_{g} \times S^{1}$ with degree $k$}

Let $K$ be an oriented knot in $S_{g} \times S^{1}$ with degree $k$. Then there exists a lifting $\hat{K}$ to $S_{g} \times \mathbb{R}$. Since $K$ has the degree $k$, $\hat{K}$ is not a knot in $S_{g} \times \mathbb{R}$, but it is a long knot of infinite copy of $\hat{K} \cap [0,k]$. But, if we consider a covering $p_{k}$ from $S_{g} \times S^{1}$ to $S_{g} \times S^{1}$ defined by $p_{k}(x,z) = (x,z^{k})$, the lifting $\hat{K}$ to $S_{g} \times S^{1}$ becomes a knot in $S_{g} \times S^{1}$. The algorithm to obtain is similar to the algorithm in the previous section:\\
\textbf{Step 1.} Let $D$ be an oriented diagram of $K$ in $S_{g} \times S^{1}$. Let fix a point point on a diagram and give a label for each arcs according to double lines. Let us say that the minimal label is $m$ and the maximal label is $M$, where $M-m = k$.\\

\textbf{Step 2.} Let us make $M-m+1$ parallel planes placed vertically. Give numberings from bottom to top by integers from $m$ to $M$. Draw $M-m$ copies of $D$ for each plane. For a copy of a diagram on the plane with $k$ erase arcs which are not labeled by $k$. \\

\textbf{Step 3.} Let us start walking from the point on the diagram on the plane with number $0$, which corresponds to the fixed point. When we meet the double line, if it is longer line, then we connect the arc to the arc on the plane with number $1$, but if it is shorter line, then we connect the arc to the arc on the plane with number $-1$. Let us denote the obtained diagram by $\hat{D}$. Note that, since the degree is $k$, we need to connect some arcs on the highest ($M$-th) layer with arcs on the lowest ($m$-th) layer, although the arcs go up according to double lines.
When we connect them, we add a double line on the connecting arc. Let us denote the obtained diagram by $\hat{D}^{k}$.

\begin{cor}
Let $D$ and $D'$ be two oriented diagrams. If they are equivalent, then $\hat{D}$ and $\hat{D'}$ are equivalent as knots in $S_{g} \times S^{1}$ with degree $1$ or $-1$.
\end{cor}

If we consider liftings $\hat{K}_{s}$ such that $\hat{K}_{s}(0) = \hat{K}_{s}(0) = (x, e^\frac{2\pi si}{k})$, then we obtain a link in $S_{g} \times S^{1}$ of $k$ components. Note that the classical crossing with label $b-a$, then it corresponds to the crossing of infinite link $\sqcup_{s = 0, \dots, k-1} K_{s}$ between two components $K_{a}$ and $K_{b}$.
 
\section{Links in $S^{3}$ and links in $S_{g} \times S^{1}$}

Let us remind that one can obtain a link $K$ in $\Sigma \times S^{1}$ by using 2-component links $K \sqcup J$ in $S^{3}$, where $J$ is a fibered knot (or link) and $\Sigma$ is a Seifert surface of $J$. In this section, we will discuss how to obtain a link $K \sqcup J$ in $S^{3}$ from a link $K$ in $S_{g} \times S^{1}$. \\[2mm]

\textbf{Step 1.} Let $K$ be a link $S_{g} \times S^{1}$ and $D$ a diagram of $D$.\\

 \begin{figure}[h!]
\begin{center}
 \includegraphics[width = 5cm]{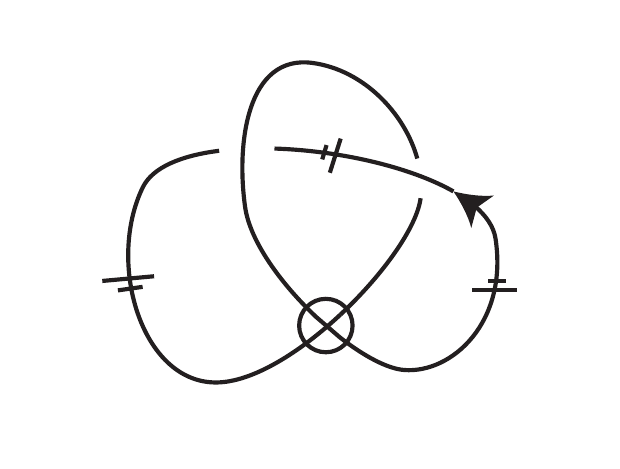}
\end{center}
\caption{A diagram of a knot in $S_{g} \times S^{1}$ of degree $3$}
\end{figure}

\textbf{Step 2.} Let us construct band presentation of the diagram.\\

 \begin{figure}[h!]
\begin{center}
 \includegraphics[width = 5cm]{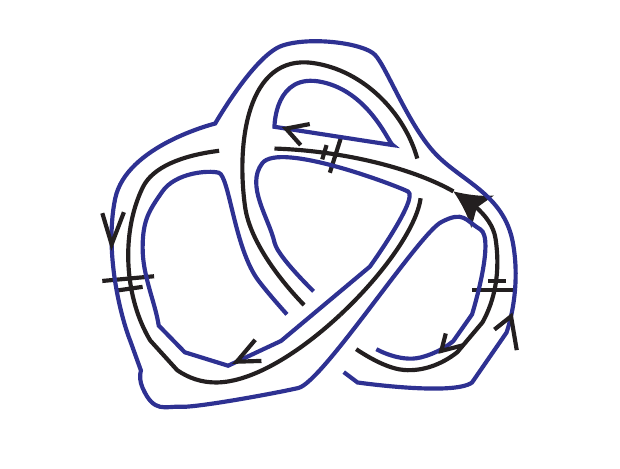}
\end{center}
\caption{A band presentation of $D$}\label{exa_linkwfiber-2}
\end{figure}

\textbf{Step 3.} To obtain an orientable surface, we twist the band where source-sink structure is broken, see Fig.\ref{exa_linkwfiber-3}. This is always possible and it is proved in \cite{NaokoKamada}.

 \begin{figure}[h!]
\begin{center}
 \includegraphics[width = 5cm]{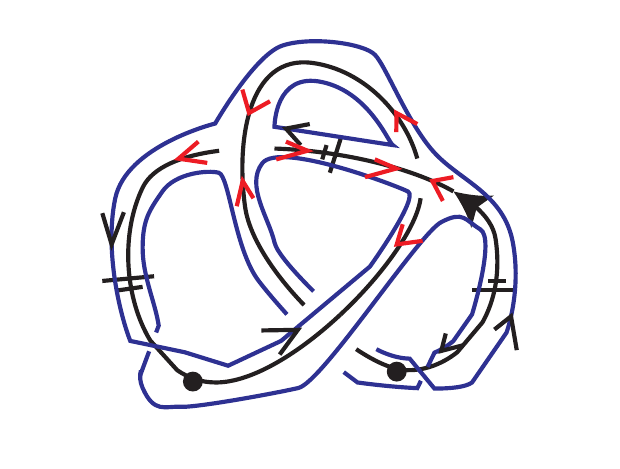}
\end{center}
\caption{A band presentation of $D$}\label{exa_linkwfiber-3}
\end{figure}

\textbf{Step 4.} If there is a circle, which can be contractable, then we connect it with one of boundary of bands as described in Fig.~\ref{exa_linkwfiber-rule}. For each double line we make a link as described in Fig.~\ref{exa_linkwfiber-rule}. Then we obtain a link $K \sqcup J$, where $J$ is a fibered link, see Fig. \ref{exa_linkwfiber-4}.

 \begin{figure}[h!]
\begin{center}
 \includegraphics[width = 8cm]{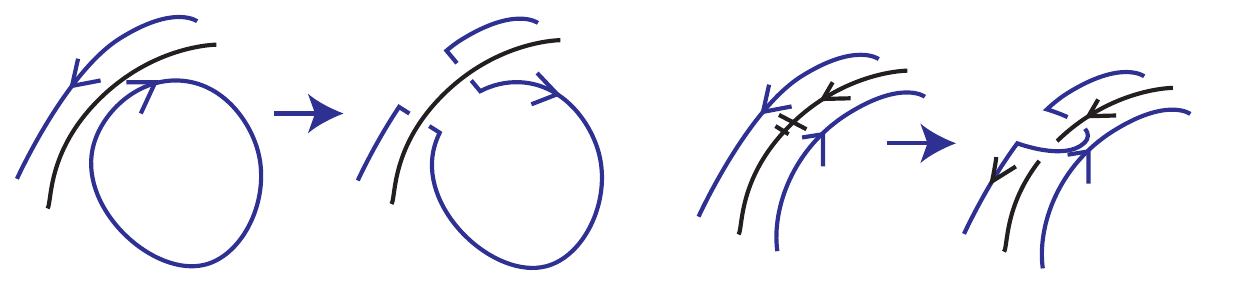}
\end{center}
\caption{}\label{exa_linkwfiber-rule}
\end{figure}

 \begin{figure}[h!]
\begin{center}
 \includegraphics[width = 5cm]{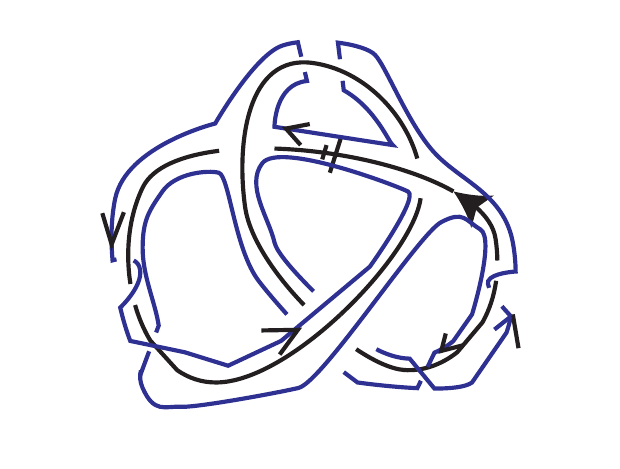}
\end{center}
\caption{A link $K \sqcup J$, where $J$ is a fibered link}\label{exa_linkwfiber-4}
\end{figure}

If $K$ is a knot, then it is easy to see that $lk(K, J) = deg(K)$. Moreover, the knot $K$ is placed on a Seifert surface $\Sigma$ with $\partial \Sigma = J$.

\textbf{Question} The component $J$ of link $K \sqcup J$ is a fibered link?\\[2mm]

\textbf{Answer} Yes. \\[2mm]

\textbf{Question} Is this mapping well-defined? \\[2mm]

\textbf{Answer} No.

\begin{exa}
In Fig.~\ref{exa_notgood} we have two diagrams of links in $S_{g} \times S^{1}$. But, the links in $S^{3}$ obtained from them are not equivalent, more precisely, the left link in Fig.~\ref{exa_notgood} is a Hopf link with separated trivial component, but  the left link in Fig.~\ref{exa_notgood} is a Borromean link.
 \begin{figure}[h!]
\begin{center}
 \includegraphics[width = 4cm]{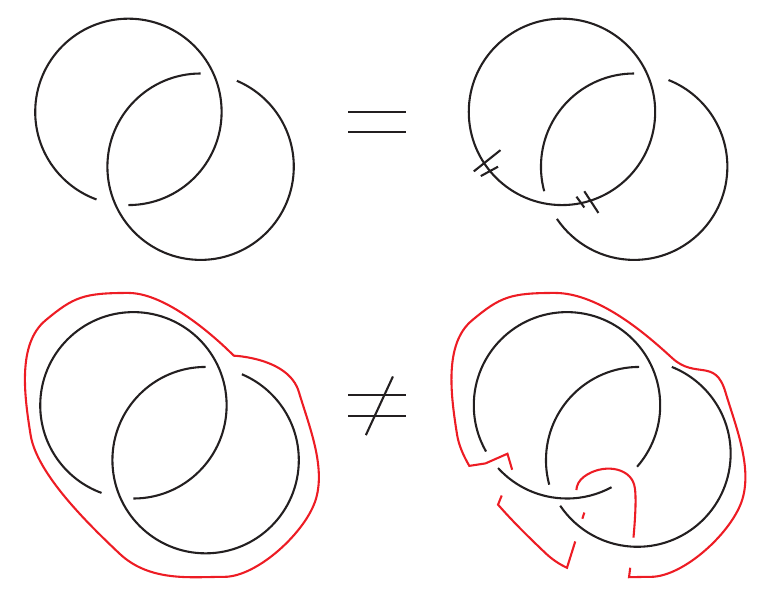}
\end{center}
\caption{A link in $S_{g} \times S^{1}$ corresponding to two different links}\label{exa_notgood}
\end{figure}
\end{exa}

\section{Simple invariants for knots in $S^{2} \times S^{1}$}

\subsection{Unknotting number of knots in $S^{2} \times S^{1}$}

Let us consider a knot in $S^{2} \times S^{1}$. Let $D$ be a diagram of $K$. Notice that $D$ has no virtual crossings. Now let us consider ``crossing change'' of the diagram $D$ as Fig.~\ref{crossing_change}.

  \begin{figure}[h!]\label{crossing_change}
\begin{center}
 \includegraphics[width = 4cm]{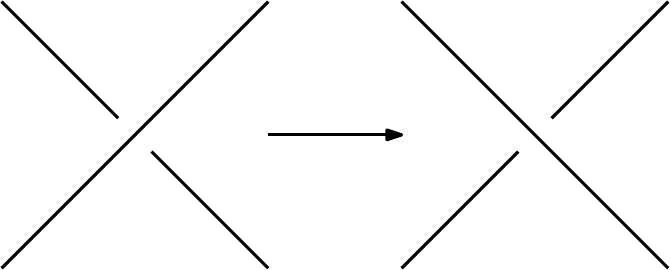}

\end{center}
 \caption{}
\end{figure}

For diagrams of knots (or links) in $S^{2} \times S^{1}$ the operation ``crossing change'' gives us an unknotting operation as follow:

Step 1) By the move (4) in Fig.~\ref{moves2} and crossing change, we can move every double lines to one semi-arc, see Fig.~\ref{cross_double_line}.

  \begin{figure}[h!]\label{cross_double_line}
\begin{center}
 \includegraphics[width = 7cm]{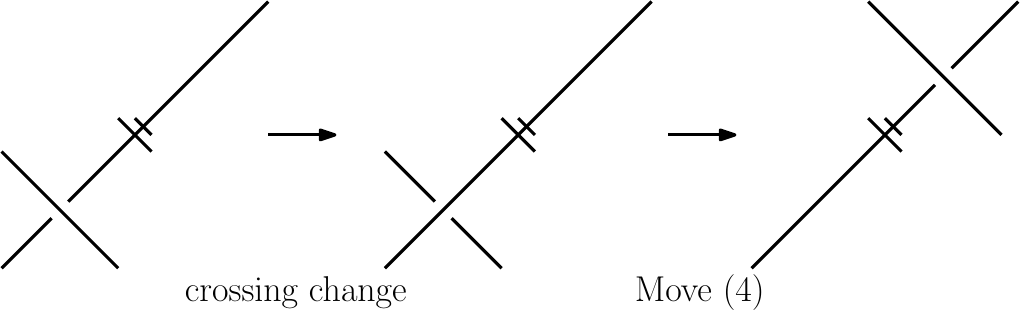}

\end{center}
 \caption{}
\end{figure}

Step 2) After that we obtain a diagram, in which every double line is on a semi-arc. Then it is easy to see that we can transform it to a descending diagram.

Now we call the minimum number of crossing change to obtain a trivial diagram {\it unknotting number}. It must be an invariant for knots in $S^{2} \times S^{1}$.

\subsection{Quandle like invariants for knots in $S_{g}\times S^{1}$}

\begin{dfn}
  Let $Q$ be a set with binary operations $(\{*^{\pm}_{i}\},\{\circ^{\pm}_{j}\})$ for some positive integer $n$ and a function $S$ from $Q$ to itself. If $(Q,\{*^{\pm}_{i}\},\{\circ^{\pm}_{j}\},S)$ satisfies the following properties, we call this a \textit{labeled quandle} of degree $n$.
  \begin{enumerate}
  \item $x*_{0} x = x \circ_{0} x$
  \item $x*_{j}y = x*_{j}(y*_{j}x)$, $(x \circ_{i} y)\circ_{i} y=x$
  \item $x*_{i}y=x*_{i}(y \circ_{i}^{-1} x)=x\circ_{j}(S(y))$,\\
  $S(y)*_{j} x=S(y\circ_{i}^{-1}x)$,\\
  $x\circ_{j}^{-1}S(y\circ_{i}^{-1}x)=x\circ_{i}^{-1}(y\circ_{i}^{-1}x)$,\\
  $S(y \circ_{i}^{-1}x)*_{j}(x \circ_{j}^{-1}S(y\circ_{i}^{-1}x))=S(y)$. $i+j =\pm 1$.
  \item $(z\circ_{i}y)\circ_{k}x=(z\circ_{k}(x\circ_{j}y)\circ_{i}(y\circ_{j}x)$, \\ $(y\circ_{j}x)*_{i}(z\circ_{k}(x*_{j}y))=(y*_{j}z)\circ_{j}(x*_{k}(z*_{i}y))$,\\
       $(x*_{j}y)*_{k}z = (x*_{k}(z\circ_{i}y))*_{j}(y*_{i}z)$. $i+j-k =0$.
\end{enumerate}
\end{dfn}

\begin{lem}
  Let $(Q,\{*^{\pm}_{i}\},\{\circ^{\pm}_{j}\},S)$ be a labeled quandle. For all $i$, $S(a*_{i}b)=S(a)*_{i}S(b)$ and $S(a\circ_{i}b)=S(a)\circ_{i}S(b)$.
\end{lem}
\begin{proof}
For some $i,j$ with $i+j =0$, from the equation (3)-4, we can get $S(t)*_{j}(x \circ_{j}^{-1}(S(t)))=S(t \circ_{i} x)$ by putting $t= y \circ_{j}^{-1} x$. From the equation (3)-1, we can get $S(t)*_{j}(x \circ_{j}^{-1}(S(t))) = S(t) \circ_{i} S(y)$. Therefore, $S(t \circ_{i} x)=S(t) \circ_{i} S(y)$. Note that from the equation (2), $\circ = \circ^{-1}$. From the equation (3)-1 and (3)-2, $S(x) *_{i}S(y) = S(x) \circ_{j}S(S(y))$ $= S(x) \circ_{j}^{-1} S(S(y)) = S(x*_{i}y)$.
\end{proof}

\begin{lem}
Let $D$ be a labeled diagram with degree $n$. Let $(Q,\{*^{\pm}_{i}\},\{\circ^{\pm}_{j}\},S)$ be a labeled quandle of degree $n$. A coloring for $D$ by $(Q,\{*^{\pm}_{i}\},\{\circ^{\pm}_{j}\},S)$ is invariant under the moves for knots in $S_{g} \times S^{1}$.
\end{lem}
\begin{proof}
  By definition, this is straight forward.
\end{proof}

\begin{rem}
   Let $(Q,\{*^{\pm}_{i}\},\{\circ^{\pm}_{j}\},S)$ be a labeled quandle. Then $(Q, *_{0}, \circ_{0})$ is a biquandle where $B : Q \times Q \rightarrow Q \times Q$ is defined by $B(x,y) = (y\circ_{0}x, x*_{0}y)$ and $B^{-1}(x,y) = (y\circ_{0}^{-1}x, x*_{0}(y\circ_{0}^{-1}x))$.
   \begin{figure}[h!]
\begin{center}
 \includegraphics[width = 8cm]{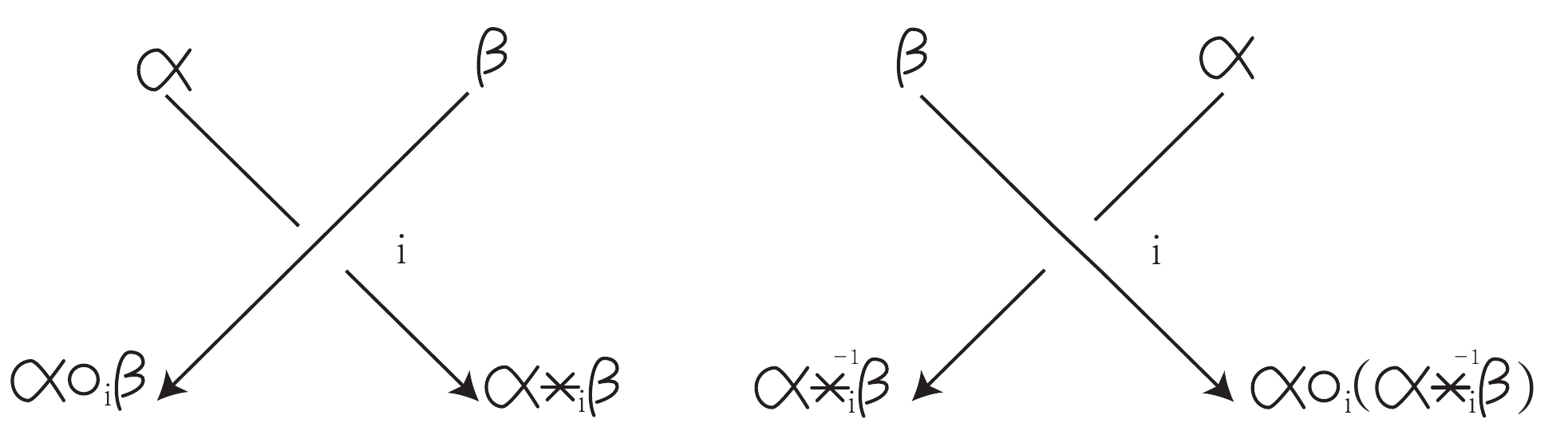}

\end{center}
 \caption{}
\end{figure}
\end{rem}

We may define free labeled quandles and we can define labeled quandles from knots in $S_{g} \times S^{1}$. But we have NO REAL EXAMPLES for labeled quandles. Moreover, we do not know what we can know from the quandle. We need to find real examples from already know algebraic structures.

\subsection{Bracket polynomial for labeled diagrams}

Let us define a polynomial invariant for links in $S_{g} \times S^{1}$. For a diagram $D_{L}$ of a link $L$ in $S_{g} \times S^{1}$ we define $[D_{L}]$ valued in $\mathbb{Z}[A_{a,b}^{\pm 1}, B_{a,b}^{\pm 1}, \delta_{m}]$ for $a,b \in \mathbb{Z}$ and $m \in \mathbb{Z}_{2}$ by skein relations as in Fig.~\ref{def_bracket}.

   \begin{figure}[h!]\label{def_bracket}
\begin{center}
 \includegraphics[width = 6cm]{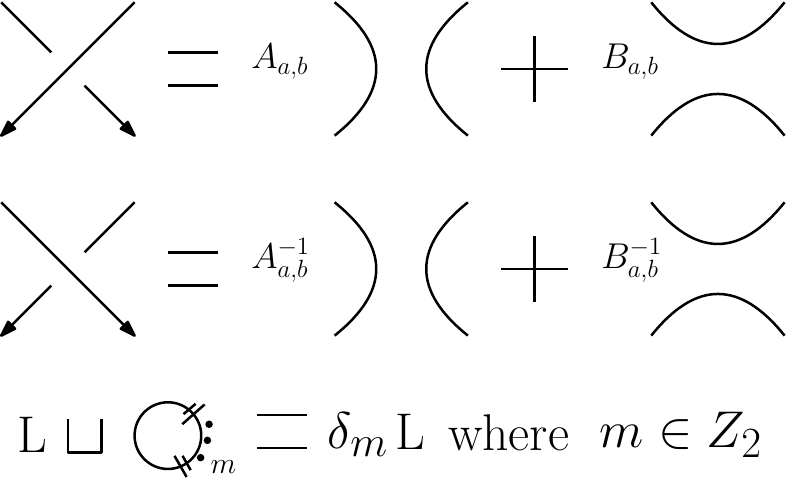}

\end{center}
 \caption{}
\end{figure}
To obtain an invariant, the following relations are needed.
$$\delta_{0}+B_{a,b}A_{a,b}^{-1}+A_{a,b}B_{a,b}^{-1} =0$$
$$A_{a,b}A_{a,c}B_{b,c}+B_{a,b}A_{a,c}A_{b,c} + B_{a,b}A_{a,c}B_{b,c} = A_{a,b}B_{a,c}A_{b,c}$$
$$A_{a,b}\delta_{i}\delta_{j} + B_{a,b}\delta_{k}\delta_{l}= C_{a,b}\delta_{i+1}\delta_{j-1} + D_{a,b}\delta_{k-1}\delta_{l+1},$$
for any $a,b,c \in \mathbb{Z}$ and $i,j,k,l \in \mathbb{Z}_{2}$.

\end{document}